\documentclass{amsart}
\usepackage{amssymb,amsthm,amsfonts,amsmath,pifont}
\usepackage{enumerate}
\usepackage{hyperref}

\title[The geometry of closed conformal vector fields]{The geometry of closed conformal vector fields on Riemannian spaces}

\author{A. Caminha}
\address{Departamento de Matem\'atica, Universidade Federal do Cear\'a, Fortaleza,
Cear\'a, Brazil. 60455-760}
\email{antonio.caminha@gmail.com}

\subjclass[2000]{Primary 53C42, 53C45; Secondary 53C65}

\keywords{Conformal vector fields; warped products; Jellett's theorem; Bernstein-type theorems}

\thanks{The author is partially supported by CNPq}

\dedicatory{To prof. M. Dajczer on occasion of his $60^{\rm th}$ birthday}

\newtheorem{theorem}{Theorem}[section]
\newtheorem{proposition}[theorem]{Proposition}
\newtheorem{lemma}[theorem]{Lemma}

\newtheorem{corollary}[theorem]{Corollary}
\newtheorem{example}[theorem]{Example}
\newtheorem{examples}[theorem]{Examples}
\newtheorem{remark}[theorem]{Remark}

\begin{document}

\begin{abstract}
In this paper we examine different aspects of the geometry of closed conformal vector fields on Riemannian manifolds. We begin by getting obstructions to the existence of closed conformal and nonparallel vector fields on complete manifolds with nonpositive Ricci curvature, thus generalizing a theorem of T. K. Pan. Then we explain why it is so difficult to find examples, other than trivial ones, of spaces having at least two closed, conformal and homothetic vector fields. We then focus on isometric immersions, firstly generalizing a theorem of J. Simons on cones with parallel mean curvature to spaces furnished with a closed, Ricci null conformal vector field; then we prove general Bernstein-type theorems for certain complete, not necessarily cmc, hypersurfaces of Riemannian manifolds furnished with closed conformal vector fields. In particular, we obtain a generalization of theorems J. Jellett and A. Barros and P. Sousa for complete cmc radial graphs over finitely punctured geodesic spheres of Riemannian space forms.
\end{abstract}

\maketitle

\section{Introduction}

Conformal vector fields are relevant objects for the geometry of several kinds of spaces. For example, if $I\subset\mathbb R$ is an open interval, $f$ is a positive smooth function on $I$ and $F$ is a Riemannian manifold, the warped product $M=I\times_fF$ carries the natural closed conformal vector field $(f\circ\pi_I)\partial_t$, where $\pi_I:M\rightarrow I$ is the canonical projection and $t$ the standard coordinate on $I$; on the other hand, if one fixes a point $p$ in a Riemannian manifold $M$ of constant sectional curvature $c$, let $r$ be the distance function from $p$ on $M$ and $s$ the solution to the ODE $y''+cy=0$, $y(0)=0$, $y'(0)=1$, then $\xi=(s\circ r)\nabla r$ is a conformal vector field out of the cut locus of $p$ in $M$, where $\nabla r$ stands for the gradient of $r$ on $M$. These occurrences of conformal vector fields are two that we shall explore in this paper; as for a third one (cf.~\cite{Castro:01}), let us mention that if $M$ is a pseudoumbilic submanifold of a complex space form $(\overline{\mathbb M},J)$, and $H$ is the  mean curvature vector of $M$ on $\overline{\mathbb M}$, then $JH$ is a conformal vector field on $M$. 

Due to the range of examples above, it is no surprise that closed conformal vector fields have been taken as natural tools in the geometric study of submanifolds and, more particularly, of hypersurfaces. To mention just some recent work related to this paper, A. Barros, A. Brasil Jr. and the author~\cite{Barros:08} used the natural closed conformal vector field on a Lorentz warped product with $1-$dimensional basis to classify, under appropriate conditions, strongly stable, closed cmc spacelike hypersurfaces of such a space as being either maximal or spacelike slices; H. F. de Lima and the author~\cite{Caminha:091} studied complete cmc vertical graphs in the hyperbolic and steady state spaces, obtaining necessary conditions for their existence under natural growth restrictions on the corresponding height functions; A. Barros and P. Sousa~\cite{Barros:082} applied the conformal vector field on constant curvature spaces to extend to the sphere a theorem of J. Jellett~\cite{Jellett:1853} on cmc closed radial graphs of $\mathbb R^{n+1}$. 

In this paper we focus on three different aspects of conformal vector fields on Riemannian spaces, namely, obstructions to their existence, their use to generate isometric immersions with parallel mean curvature and their role in Bernstein-type results. We begin, in section~\ref{section:Conformal vector fields}, refining a result of S. T. Yau~\cite{Yau:76} (cf. Proposition~\ref{prop:first corollary of Yau 76}), which is then applied to the study of closed conformal vector fields on complete, noncompact, Riemannian manifolds with nonpositive Ricci curvature; according to Proposition~\ref{prop:criterion for parallelism}, provided a suitable integrability condition on the vector field is satisfied, we show the vector field has to be both parallel and a direction of vanishing of the Ricci curvature on the manifold. This result generalizes a previous theorem of T. K. Pan~\cite{Pan:63}. Continuing this study, in Lemma~\ref{lemma:constructing closed conformal fields on leaves} of section~\ref{section:Homothetic fields} we give a general geometric construction of closed conformal vector fields on certain hypersurfaces of ambient spaces furnished with two closed conformal and homothetic vector fields; this is then applied, in Theorem~\ref{thm:spaces with two li homothetic vector fields}, to explain why it is not easy to give examples, out of the trivial ones, of manifolds with many homothetic vector fields.

In section~\ref{section:On a result of J. Simons} we generalize a theorem os J. Simons~\cite{Simons:68} on cones with parallel mean curvature on Euclidean spaces spaces furnished with a closed conformal vector field. More precisely, we show that if the vector field is Ricci null, then, by evolving any minimal immersion on a leaf of the perp distribution of the vector field, we get a submanifold of the ambient space with parallel mean curvature.

Bernstein-type results are dealt with in section~\ref{section:Bernstein-type theorems}. For ambient spaces of nonnegative Ricci curvature and possessing a parallel and a homothetic nonparallel vector field, Theorem~\ref{thm:Bernstein-type} establishes a reasonable set of sufficient conditions for a complete, oriented (not necessarily cmc) hypersurface of bounded second fundamental form to be totally geodesic and, more particularly, a leaf of the orthogonal foliation determined by the parallel vector field. This result generalizes our previous work in~\cite{Caminha:09}. Next, in Theorem~\ref{thm:Bernstein-type 2} we consider again an ambient space of nonnegative Ricci curvature,  this time with a single homothetic vector field; under the same set of necessary conditions, we establish the umbilicity of a complete, oriented cmc hypersurface of bounded second fundamental form. We
specialize this result to $\mathbb R^{n+1}\setminus\{0\}=(0,+\infty)\times_t\mathbb S^n$ in Corollary~\ref{coro:Jellett's theorem over the punctured sphere}, thus generalizing of a theorem of J. Jellett~\cite{Jellett:1853} for complete cmc radial graphs over a finitely punctured sphere in Euclidean space. We point out that the existence of a homothetic vector field on a manifold gives rise to a non-parallel calibration in the sense of R. Harvey and B. Lawson~\cite{Harvey:82}; in this respect, G. Li and I. Salavessa~\cite{Li:09} obtained some Bernstein-type results related to ours (though not the same ones). Finally, Theorem~\ref{thm:Bernstein-type 3} deals with the corresponding situation in ambients of constant sectional curvature, relying this time on the conformal vector field $\xi=(s\circ r)\nabla r$. Using techniques similar to the ones required for our previous results, Corollary~\ref{coro:Jellett's theorem for spheres and hyperbolic} generalizes the main theorem of~\cite{Barros:082} for complete cmc radial graphs over a finitely punctured geodesic sphere of the Euclidean sphere and the hyperbolic space.

\section{Obstructions on the existence of conformal fields}\label{section:Conformal vector fields}

Let $M^n$, $n\geq 2$, be an $n-$dimensional, oriented Riemannian manifold with metric $g=\langle\,\,,\,\rangle$ and Levi-Civita connection $\nabla$. We recall that a smooth vector field $\xi$ on $M$ is conformal if
$$\mathcal L_{\xi}g=2\psi_{\xi}g$$
for some smooth function $\psi_{\xi}$ on $M$, where $\mathcal L_{\xi}$ denotes the Lie derivative in the direction of $\xi$. The function $\psi_{\xi}$ is the conformal factor of $\xi$.

Since $\mathcal L_{\xi}(X)=[\xi,X]$ for every $X\in\mathfrak X(M)$, it follows from the tensorial character of Lie derivatives that $\xi\in\mathfrak X(M)$ is conformal if and only if there exists a smooth function $\psi_{\xi}$ on $M$ such that
$$\langle\nabla_X\xi,Y\rangle+\langle X,\nabla_Y\xi\rangle=2\psi_{\xi}\langle X,Y\rangle$$
for all $X,Y\in\mathfrak X(M)$. In particular,
\begin{equation}\label{eq:conformal factor as a divergence}
 \psi_{\xi}=\frac{1}{n}{\rm div}_M\,\xi.
\end{equation}

Therefore, if $M$ is closed, divergence theorem assures that a necessary condition for the existence of such a vector field $\xi$ is that
$$\int_M\psi_{\xi}dM=0.$$

An interesting particular case of a conformal vector field $\xi$ is that in which $\nabla_X\xi=\psi_{\xi}X$ for all $X\in\mathfrak X(M)$; in this case we say that $\xi$ is closed, an allusion to the fact that its dual $1-$form $\omega_{\xi}$ is closed. Yet more particularly, a closed and conformal vector field $\xi$ is said to be parallel if $\psi_{\xi}$ vanishes identically, and homothetic if $\psi_{\xi}$ is constant\footnote{Here we diverge a little bit from other papers (e.g.~\cite{Kuhnel:97}), where homothetic means just conformal with constant conformal factor. The reason is economy: to avoid constantly writing {\em closed and homothetic}.}.

It is easy to prove that the conformal character of a vector field is invariant under a conformal change of metric. On the other hand, a closed conformal vector field remains closed if and only if the conformal change has constant conformal factor, in which case the conformal factor of the field does not change.

From now on, unless stated otherwise, we stick to closed conformal fields. 

Suppose $\xi\in\mathfrak X(M)$ is a nontrivial closed conformal vector field on $M$, with conformal factor $\psi_{\xi}$, and let $\omega=\xi^*$ be its dual $1-$form. Weitzenböck's formula (cf. Proposition $1.3.4$ of~\cite{Xin:96}) gives us
\begin{equation}\label{eq:auxiliar 1 para campos conformes fechados em variedades hiperbolicas}
 \Delta\omega=-\nabla^2\omega+\text{Ric}_M(\xi),
\end{equation}
where $\nabla^2$ is the trace-Laplacian and $\Delta=d\delta+\delta d$ the Hodge Laplacian on $\Lambda^1(M)$, and $\text{Ric}_M(\xi)\in\mathfrak X(M)$ is such that
$$\langle\text{Ric}_M(\xi),\xi\rangle=\text{Ric}_M(\xi,\xi),$$
the Ricci tensor of $M$ evaluated at $\xi$.

Now, let $(\cdot,\cdot)$ stand for the $\mathcal L^2$ inner product on $\Lambda^1M$. Since $\omega$ is closed and $d$ and $\delta$ are adjoint with respect to $(\cdot,\cdot)$, we have $(\Delta\omega,\omega)=(\delta\omega,\delta\omega)$.
However, if $\{e_1,\ldots,e_n\}$ is a local orthonormal frame field on $M$, then\footnote{Whenever there is no danger of confusion, we use Einstein's summation convention.}
\begin{equation}\label{eq:auxiliar 2 para campos conformes fechados em variedades hiperbolicas}
\delta\omega=-(\nabla_{e_i}\omega)(e_i)=-\langle(\nabla_{e_i}\omega)^*,e_i\rangle=-\langle\nabla_{e_i}\xi,e_i\rangle=-n\psi_{\xi},
\end{equation}
so that
$$(\Delta\omega,\omega)=\int_Mn^2\psi_{\xi}^2dM.$$
Finally, it follows from the closed conformal character of $\xi$ that
\begin{equation}\label{eq:auxiliar 3 para campos conformes fechados em variedades hiperbolicas}
-(\nabla^2\omega,\omega)=\int_M\sum_i|\nabla_{e_i}\omega|^2dM=\int_M\sum_i|\nabla_{e_i}\xi|^2dM=\int_Mn\psi_{\xi}^2dM. 
\end{equation}

If $M$ has nonpositive Ricci curvature, then, substituting the above computations in (\ref{eq:auxiliar 1 para campos conformes fechados em variedades hiperbolicas}), we get
$$\int_Mn(n-1)\psi_{\xi}^2dM=\int_M\text{Ric}_M(\xi,\xi)dM\leq 0,$$
and hence $\psi_{\xi}=0$ on $M$. This way, $\nabla_Y\xi=\psi_{\xi}Y=0$ for all $Y\in\mathfrak X(M)$, and $\xi$ is parallel.


The above result was known to T. K. Pan in~\cite{Pan:63}, and we now proceed to generalize it to complete noncompact Riemannian manifolds. To this end, we begin by quoting a version of Stokes' theorem on an $n-$dimensional, complete noncompact Riemannian manifold $M$, obtained by S. T. Yau in~\cite{Yau:76}: if $\mathcal L^1(M)$ is the space of Lebesgue integrable functions on $M$, we say that $\omega\in\Omega^{n-1}(M)$ is an integrable $n-1$ differential form on $M^n$ if $|\omega|\in\mathcal L^1(M)$. If this is the case, then there exists a sequence $B_i$ of domains on $M$, such that $B_i\subset B_{i+1}$, $M=\cup_{i\geq 1}B_i$ and
$$\lim_{i\rightarrow+\infty}\int_{B_i}d\omega=0.$$

Suppose $M$ is oriented by the volume element $dM$. If $\omega=\iota_XdM$ is the contraction of $dM$ in the direction of a smooth vector field $X$ on $M$,
then one gets the following consequence of Yau's result.

\begin{proposition}\label{prop:first corollary of Yau 76}
Let $X$ be a smooth vector field on the complete, noncompact, oriented Riemannian manifold $M^n$, such that ${\rm div}_MX$ doesn't change sign on $M$. If $|X|\in\mathcal L^1(M)$, then ${\rm div}X=0$ on $M$.
\end{proposition}

\begin{proof}
Assume, without loss of generality, that ${\rm div}_MX\geq 0$ on $M$. If $\{e_1,\ldots,e_n\}$ is an orthonormal frame on an open set $U\subset M$,
with coframe $\{\omega_1,\ldots,\omega_n\}$, then
$$\iota_XdM=(-1)^{i-1}\langle X,e_i\rangle\omega_1\wedge\ldots
\wedge\widehat\omega_i\wedge\ldots\wedge\omega_n.$$

Since the $(n-1)-$forms $\omega_1\wedge\ldots\wedge\widehat\omega_i\wedge\ldots
\wedge\omega_n$ are orthonormal in $\Omega^{n-1}(M)$, we get $$|\omega|^2=
\sum_{i=1}^n\langle X,e_i\rangle^2=|X|^2.$$
Then $|\omega|\in\mathcal L^1(M)$ and $d\omega=d(\iota_XdM)=({\rm div}X)dM$. Letting $B_i$ be as in the
preceeding discussion, we get
$$\int_{B_i}({\rm div}X)dM=\int_{B_i}d\omega\stackrel{i}{\longrightarrow}0.$$

It now follows from ${\rm div}X\geq 0$ on $M$ that ${\rm div}X=0$ on $M$.
\end{proof}

\begin{remark}
With respect to the above proposition, in~\cite{Yau:76} Yau considered just the case $X=\nabla f$, for some smooth function $f:M\rightarrow\mathbb R$. Although the generalization just proved seems to be a small one, we will see that it gives greater flexibility in what concerns applications.
\end{remark}

We are now able to state and prove the following result. We remark that (\ref{eq:integrability condition}) reduces to $\psi_{\xi}\in\mathcal L^1(M)$ if $\xi$ is a bounded (hence complete) vector field.

\begin{proposition}\label{prop:criterion for parallelism}
 Let $M^n$ be an $n-$dimensional complete Riemannian manifold with nonpositive Ricci curvature, and let $\xi$ be a closed conformal vector field on $M$, with conformal factor $\psi_{\xi}$. If
\begin{equation}\label{eq:integrability condition}
\int_M|\psi_{\xi}\xi|dM<+\infty,
\end{equation}
then $\xi$ is parallel and the Ricci curvature of $M$ in the direction of $\xi$ vanishes identically.
\end{proposition}

\begin{proof}
As before, letting $\omega\in\Lambda^1(M)$ be the $1-$form dual of $\xi$, we get $d\omega=0$, and it follows from (\ref{eq:auxiliar 2 para campos conformes fechados em variedades hiperbolicas}) and (\ref{eq:conformal factor as a divergence}) that
\begin{equation}\label{eq:auxiliar 1 para criterion for parallelism}
 \begin{split}
\langle\Delta\omega,\omega\rangle&\,=\langle d\delta\omega,\omega\rangle=-n\langle d\psi_{\xi},\omega\rangle=-n\langle\nabla\psi_{\xi},\xi\rangle\\
&\,=-n\,\text{div}_M(\psi_{\xi}\xi)+n^2\psi_{\xi}^2.   
 \end{split}
\end{equation}
On the other hand, the proof of the Proposition at page $8$ of~\cite{Xin:96} gives
\begin{equation}\label{eq:auxiliar 2 para criterion for parallelism}
\langle\nabla^2\omega,\omega\rangle=\text{div}_M(\psi_{\xi}\xi)-n\psi_{\xi}^2. 
\end{equation}

Substituting (\ref{eq:auxiliar 1 para criterion for parallelism}) and (\ref{eq:auxiliar 2 para criterion for parallelism}) in the inner product of Weitzenböck's formula (\ref{eq:auxiliar 1 para campos conformes fechados em variedades hiperbolicas}) with $\omega$, we get 
\begin{equation}\label{eq:divergence of psi xi xi}
\text{div}_M(\psi_{\xi}\xi)=n\psi_{\xi}^2-\frac{1}{n-1}\text{Ric}_M(\xi,\xi).
\end{equation}
Since ${\rm Ric}_M\leq 0$, we get $\text{div}_M(\psi_{\xi}\xi)\geq 0$, and it follows from (\ref{eq:integrability condition}) and Proposition~\ref{prop:first corollary of Yau 76} that $\text{div}_M(\psi_{\xi}\xi)=0$ on $M$.
Back on (\ref{eq:divergence of psi xi xi}), we get $\psi_{\xi}=0$ and $\text{Ric}_M(\xi,\xi)=0$, as we wished to prove.
\end{proof}

An interesting class of spaces furnished with closed conformal vector fields is given by the following subclass of
warped product spaces: let $F^{n-1}$ be an $(n-1)-$dimensional Riemannian manifold with metric $(\,\,,\,\,)$,
$I\subset\mathbb R$ be an open interval and $f:I\rightarrow\mathbb R_+^*$ be a smooth function, called the
warping function. Set $M^n=I\times F^{n-1}$ (as a differentiable manifold), and let $\pi_I:M\rightarrow I$ and
$\pi_F:M\rightarrow F$ denote the canonical projections; we equip $M$ with the warped metric $\langle\,\,,\,\,\rangle$, given by
$$\langle X,Y\rangle=d\pi_I(X)d\pi_I(Y)+(f\circ\pi_I)^2(d\pi_F(X),d\pi_F(Y)),$$
and denote the resulting Riemannian space by
$$M^n=I\times_fF^{n-1}.$$
If $\partial_t$ is the standard unit vector field on $I$, then $\xi=(f\circ\pi_I)\partial_t$ is closed conformal,
with conformal factor $\psi_{\xi}=f'\circ\pi_I$ (cf.~\cite{O'Neill:83}, Proposition $7.35$).

As an easy consequence of Proposition $7.42$ of~\cite{O'Neill:83}, $M^n$ has constant sectional curvature $c$ if and only if $F$ has constant sectional curvature $k$ and the warping function $\phi$ satisfies the ODE
$$\frac{\phi''}{\phi}=-c=\frac{(\phi')^2-k}{\phi^2}.$$
In particular, Cartan's theorem on classification of Riemannian simply connected space forms gives us the following warped product models of $\mathbb R^n\setminus\{0\}$, $\mathbb H^n$ and $\mathbb S^n\setminus\{\pm p\}$ ($\simeq$ means {\em isometric to}):
\begin{enumerate}
 \item[(i)] $\mathbb R^n\setminus\{0\}\simeq(0,+\infty)\times_t\mathbb S^{n-1}$, with slices corresponding to concentric Euclidean spheres centered at the origin of $\mathbb R^n$. 
 \item[(ii)] $\mathbb H^n\simeq\mathbb R\times_{e^t}\mathbb R^n$, with slices corresponding to the horospheres $x_n$ constant in the half-space model with coordinates $(x_1,\ldots,x_n)$. 
 \item[(iii)] $\mathbb S^n\setminus\{\pm p\}\simeq(0,\pi)\times_{\sin t}\mathbb S^{n-1}$, with slices corresponding to geodesic spheres of $\mathbb S^n$ centered at $p$.
\end{enumerate}

Let $F$ be oriented by the volume element $dF$, and let $dM$ denote the corresponding volume element of $M^n=I\times_f F^{n-1}$.
Since, for a vector field $X\in\mathfrak X(F)$, we have $\langle X,X\rangle=(f\circ\pi_I)^2(X,X)$, it easily follows that
$$dM=\frac{1}{f^{n-1}}dt\wedge dF.$$

We also observe that if $F$ is complete and $I=\mathbb R$, then Lemma $7.40$ of~\cite{O'Neill:83} gives $M$ complete,
regardless of the warping function $f$. Moreover, if $f$ is constant and $F$ has nonpositive Ricci curvature, then
Corollary $7.43$ of~\cite{O'Neill:83} asserts that the same is true of $M=\mathbb R\times_fF$. In what follows we use
Proposition~\ref{prop:criterion for parallelism} to obtain a partial converse to this fact.

\begin{corollary}\label{coro:when a warped product has nonnegative Ricci curvature}
Let $n\geq 3$, $F^{n-1}$ be a closed, simply connected Riemannian manifold and $f:\mathbb R\rightarrow\mathbb R_+^*$ be a smooth, nondecreasing bounded function. If $M^n=\mathbb R\times_fF^{n-1}$ has nonpositive Ricci curvature, then one of the following happens:
\begin{enumerate}
 \item[$(a)$] $f(-\infty)=0$.
 \item[$(b)$] $f$ is constant and $F$ has nonpositive Ricci curvature.
\end{enumerate}
\end{corollary}

\begin{proof}
Let $\pi:M\rightarrow\mathbb R$ denote the canonical projection. We shall assume that $n>3$, the analysis of the case $n=3$ being totally analogous. It follows from the above discussion that
\begin{eqnarray*}
\int_M|\psi_{\xi}\xi|dM&=&\int_M((f'f)\circ\pi)dM=\int_M\left(\frac{f'}{f^{n-2}}\circ\pi\right)dtdF\\
&=&{\rm Vol}(F)\int_{-\infty}^{+\infty}\frac{f'}{f^{n-2}}dt=\frac{{\rm Vol}(F)}{3-n}f^{3-n}\Big|_{-\infty}^{+\infty}.
\end{eqnarray*}
Since $f$ is bounded, nondecreasing and positive, this is finite if and only if $f(-\infty)>0$. If this is the case,
then Proposition~\ref{prop:criterion for parallelism} and our previous discussions give $\psi_{\xi}=f'\circ\pi\equiv 0$,
so that $f$ is constant. Now, Corollary $7.43$ of~\cite{O'Neill:83} guarantees that $F$ has nonpositive Ricci curvature.
\end{proof}

\section{Spaces with several homothetic vector fields}\label{section:Homothetic fields}

We now turn to the analysis of isometric immersions in the presence of conformal vector fields. To this end it is convenient to perform a slight change in the notations, so that from this section on we consider $\tilde M^{n+1}$, $n\geq 2$, an oriented, $(n+1)-$dimensional Riemannian manifold with metric $g=\langle\,\,,\,\rangle$ and Levi-Civita connection $\tilde\nabla$. We also let $\xi$ be a fixed conformal vector field on $\tilde M$, with conformal factor $\psi_{\xi}$.

Let us first analyse the case in which $\xi$ is closed and nontrivial. If it has no singularities on an open set of $\tilde M$, let $\xi^{\bot}$ be the distribution, on this open set, of vector fields orthogonal to $\xi$. For $X,Y\in\xi^{\bot}$, we have
$$\langle[X,Y],\xi\rangle=\langle\tilde\nabla_XY-\tilde\nabla_YX,\xi\rangle=-\langle Y,\tilde\nabla_X\xi\rangle+\langle X,\tilde\nabla_Y\xi\rangle=0,$$
so that $\xi^{\bot}$ is integrable by Frobenius' theorem. We let $\Xi$ be a leaf of $\xi^{\bot}$, furnished with the induced metric, and $D$ be its Levi-Civita connection. 

If we ask $\tilde M$ to be simply connected, then $\xi=\tilde\nabla f$ for some smooth function $f$ on $\tilde M$. For a smooth curve segment $\gamma:[0,1]\rightarrow\Xi$, we get
$$\frac{d}{dt}(f\circ\gamma)(t)=\langle\tilde\nabla f,\gamma'\rangle(t)=\langle\xi,\gamma'\rangle(t)=0,$$
so that $f$ is constant along $\gamma$, and hence on $\Xi$. Therefore, $\Xi$ is a connected component of a level set of $f$.

Back to a general $\tilde M$, for $X\in\mathfrak X(\Xi)$ the closedness of $\xi$ gives
$$X\langle\xi,\xi\rangle=2\langle\tilde\nabla_X\xi,\xi\rangle=2\langle\psi_{\xi}X,\xi\rangle=0,$$
i.e., $\langle\xi,\xi\rangle$ is constant on $\Xi$.
By rescaling the metric of $\tilde M$ by a constant factor we can further ask that
$\langle\xi,\xi\rangle=1$ on $\Xi$, so that $\xi$ is a global unit normal vector field on $\Xi$.
Therefore, if $S_{\Xi}$ denotes the shape operator of $\Xi$ with respect to $\xi$, we get
\begin{equation}\label{eq:shape operator of a leaf in the perp distribution}
S_{\Xi}(X)=-\tilde\nabla_X\xi=\psi_{\xi}X,
\end{equation}
and hence $\Xi$ is an umbilical hypersurface of $\tilde M$ (totally geodesic if $\xi$ is parallel).

With the above notations, we have the following 

\begin{lemma}\label{lemma:constructing closed conformal fields on leaves}
If $\eta$ is another closed conformal vector field on $\tilde M$ and $U=\eta-\langle\eta,\xi\rangle\xi$, then $U$ is closed conformal on $\Xi$, with conformal factor $\psi_U=\psi_{\eta}-\psi_{\xi}\langle\eta,\xi\rangle$.
\end{lemma}

\begin{proof}
For $Z\in\mathfrak X(\Xi)$, it follows from $\langle Z,\xi\rangle=0$ that
\begin{eqnarray}\label{eq:conformality of the projected field}
D_ZU&=&(\tilde\nabla_ZU)^{\top}=\tilde\nabla_ZU-\langle\tilde\nabla_ZU,\xi\rangle\xi\\
&=&\tilde\nabla_Z(\eta-\langle\eta,\xi\rangle\xi)-\langle\tilde\nabla_Z(\eta-\langle \eta,\xi\rangle\xi),\xi\rangle\xi\nonumber\\
&=&\tilde\nabla_Z\eta-Z\langle\eta,\xi\rangle\xi-\langle \eta,\xi\rangle\tilde\nabla_Z\xi-\langle\tilde\nabla_Z\eta,\xi\rangle\xi\nonumber\\
&&+Z\langle\eta,\xi\rangle\langle\xi,\xi\rangle\xi+\langle \eta,\xi\rangle\langle\tilde\nabla_Z\xi,\xi\rangle\xi\nonumber\\
&=&\tilde\nabla_Z\eta-\langle\eta,\xi\rangle\tilde\nabla_Z\xi-\langle\tilde\nabla_Z\eta,\xi\rangle\xi\nonumber\\
&=&\psi_{\eta}Z-\langle\eta,\xi\rangle\psi_{\xi}Z-\langle\psi_{\eta}Z,\xi\rangle\xi\nonumber\\
&=&(\psi_{\eta}-\langle\eta,\xi\rangle\psi_{\xi})Z,\nonumber
\end{eqnarray}
\end{proof}

The above Lemma allows us to give the following important

\begin{examples}
We now show how to geometrically build closed conformal vector fields on $\mathbb S^n$ and $\mathbb H^{n+1}$.
\begin{enumerate}
\item[$(a)$] For $\mathbb S^n$, let us take $\xi(x)=x$ in $\mathbb R^{n+1}$ (which is homothetic) and a parallel $\eta$; constructing $U\in T\mathbb S^n$ as above, we get $U$ closed conformal in $\mathbb S^n$.
\item[$(b)$] For $\mathbb H^n$ we take the hyperquadric model
$$\mathbb H^n=\{x\in\mathbb L^{n+1};\,\langle x,x\rangle=-1\},$$
where $\mathbb L^{n+1}$ is the $(n+1)-$dimensional Lorentz space with its usual scalar product $\langle\,\cdot\,,\,\cdot\,\rangle$, viewed as a semi-Riemannian metric. Letting $\tilde\nabla$ be the corresponding Levi-Civita connection, it is straightforward to extend the notion of conformality to this setting. Therefore, if we take $\xi(x)=x$ and a parallel $\eta$, then $\xi$ is homothetic in $\mathbb L^{n+1}$; building $U\in T\mathbb H^n$ from $\xi$ and $\eta$ as above, we get $U$ closed and conformal\footnote{Actually, this procedure gives closed conformal vector fields on all hyperquadrics of $\mathbb L^{n+1}$.}.
\end{enumerate}
\end{examples}

We need another lemma.

\begin{lemma}\label{lemma:ambient curvature along planes containing xi}
 Let $\tilde M^{n+1}$, $n\geq 3$, be a Riemannian manifold furnished with a homothetic vector field $\xi$, and let $\Xi$ be a leaf of $\xi^{\bot}$. If $R_{\tilde M}$ stands for the curvature operator of $\tilde M$, then
$$R_{\tilde M}(\xi,X)\xi=0,\,\,\forall\,X\in\mathfrak X(\Xi).$$
\end{lemma}

\begin{proof}
 The proof is an easy computation, taking into account that $\psi_{\xi}$ is constant:
\begin{eqnarray*}
R_{\tilde M}(\xi,X)\xi&=&\tilde\nabla_X\tilde\nabla_{\xi}\xi-\tilde\nabla_{\xi}\tilde\nabla_X\xi+\tilde\nabla_{[\xi,X]}\xi\\
&=&\tilde\nabla_X(\psi_{\xi}\xi)-\tilde\nabla_{\xi}(\psi_{\xi}X)+\psi_{\xi}[\xi,X]\\
&=&\psi_{\xi}\{\tilde\nabla_X\xi-\tilde\nabla_{\xi}X+[\xi,X]\}=0.
\end{eqnarray*}
\end{proof}

For what comes next, recall that the metric $\langle\,\cdot\,,\,\cdot\,\rangle$ of $\Xi^n$ is Einstein (or that $\Xi$ itself is Einstein) if there exists a smooth $\lambda:\Xi\rightarrow\mathbb R$ such that
$${\rm Ric}_{\Xi}(X,Y)=\lambda\langle X,Y\rangle,$$
for all $X,Y\in\mathfrak X(\Xi)$, where ${\rm Ric}_{\Xi}(\,\cdot\,,\,\cdot\,)$ denotes the Ricci tensor of $\Xi$. In particular, the Ricci map of $\Xi$ satisfies
\begin{equation}\label{eq:Ricci map on Einstein manifolds}
{\rm Ric}_{\Xi}(X)=\lambda X,\,\,\forall\,X\in\mathfrak X(\Xi).
\end{equation}
Moreover, if $n\geq 3$, it is a standard fact that $\lambda$ is constant on $\Xi$, so that $\Xi^n$ has constant Ricci curvature, equal to $\frac{\lambda}{n-1}$.

We now have the following result, which explains why it is not easy to give examples of manifolds with many linearly independent homothetic vector fields.

\begin{theorem}\label{thm:spaces with two li homothetic vector fields}
 Let $\tilde M^{n+1}$, $n\geq 3$, be a Riemannian manifold possessing two homothetic vector fields $\xi$ and $\eta$, with $\xi$ being non-parallel, and $\Xi$ be a complete leaf of $\xi^{\bot}$. If $\Xi$ is an Einstein manifold with positive Ricci curvature, and $\xi$ and $\eta$ are linearly independent in at least one point of $\Xi$, then:
\begin{enumerate}
 \item[$(a)$] $\Xi$ is isometric to an Euclidean $n-$sphere.
 \item[$(b)$] There exist $\epsilon>0$ and a neighborhood $\Omega$ of $\Xi$ in $\tilde M$ such that $\Omega$ is isometric to the warped product $(-\epsilon,\epsilon)\times_f\Xi$, where $f(t)=e^{-\psi_{\xi}t}$.
\end{enumerate}
\end{theorem}

\begin{proof}
 If $U=\eta-\langle\eta,\xi\rangle\xi$, then Lemma~\ref{lemma:constructing closed conformal fields on leaves} asserts that $U$ is closed conformal on $M$. Moreover, $U$ is non-homothetic; in fact, again from that lemma, it suffices to prove that $\langle\eta,\xi\rangle$ is non-constant on $M$. By contradiction, suppose it were constant and let $X\in\mathfrak X(\Xi)$. Then
\begin{eqnarray*}
0=X\langle\eta,\xi\rangle&=&\langle\tilde\nabla_X\eta,\xi\rangle+\langle\eta,\tilde\nabla_X\xi\rangle\\
&=&\psi_{\eta}\langle X,\xi\rangle+\psi_{\xi}\langle\eta,X\rangle\\
&=&\psi_{\xi}\langle\eta,X\rangle,
\end{eqnarray*}
so that $\langle\eta,X\rangle=0$. But since this is true for all $X$, we get $\eta\in T\Xi^{\bot}$, and hence $\eta$ and $\xi$ are linearly dependent everywhere on $\Xi$, which is a contradiction.

Since $n\geq 3$, the Einstein manifold $\Xi$ has constant Ricci curvature, hence constant and positive from our hypothesis; since it is complete, Bonnet-Myers theorem gives $\Xi$ compact. In particular, $U$ is complete on $\Xi$, and Theorem $3.1$ (ii) of~\cite{Kuhnel:97} implies that $\Xi$ is isometric to an Euclidean $n-$sphere.

Concerning (b), let $\Phi$ be the flow of $\xi$. The compactness of $\Xi$ allows us to choose $\epsilon>0$ such that $\Phi$ is defined in $(-\epsilon,\epsilon)\times\Xi$. For a fixed $p\in\Xi$, let $\alpha:(-\delta,\delta)\rightarrow\Xi$ be a curve such that $\alpha(0)=p$. Since
\begin{equation}\label{eq:auxiliar para calculo do raio da esfera}
\langle\frac{D}{\partial t}\Phi(t,\alpha(s)),\frac{D}{\partial s}\Phi(t,\alpha(s))\rangle\Big|_{t=0}=\langle\xi(\alpha(s)),\frac{d\alpha}{ds}\rangle=0,
\end{equation}
we can suppose that $\epsilon$ was so chosen that $\Phi:(-\epsilon,\epsilon)\times\Xi\rightarrow\tilde M$ is an embedding.

We now claim that the inner product on the left hand side of (\ref{eq:auxiliar para calculo do raio da esfera}) is zero for all $|t|<\epsilon$, and not just $t=0$. In fact, setting $\varphi(t,s)=\langle\frac{D}{\partial t}\Phi(t,\alpha(s)),\frac{D}{\partial s}\Phi(t,\alpha(s))\rangle$, the conformal character of $\xi$ gives
\begin{eqnarray*}
\frac{\partial\varphi}{\partial t}(t,s)&=&\frac{d}{dt}\langle\frac{D}{\partial t}\Phi(t,\alpha(s)),\frac{D}{\partial s}\Phi(t,\alpha(s))\rangle\\
&=&\langle\frac{D^2}{\partial t^2}\Phi(t,\alpha(s)),\frac{D}{\partial s}\Phi(t,\alpha(s))\rangle+\langle\frac{D}{\partial t}\Phi(t,\alpha(s)),\frac{D}{\partial s}\frac{D}{\partial t}\Phi(t,\alpha(s))\rangle\\
&=&\langle\frac{D}{\partial t}\xi(\Phi(t,\alpha(s))),\frac{D}{\partial s}\Phi(t,\alpha(s))\rangle+
\langle\xi(\Phi(t,\alpha(s))),\frac{D}{\partial s}\xi(\Phi(t,\alpha(s)))\rangle\\
&=&\psi_{\xi}\langle\xi(\Phi(t,\alpha(s))),\frac{D}{\partial s}\Phi(t,\alpha(s))\rangle+
\psi_{\xi}\langle\xi(\Phi(t,\alpha(s))),\frac{D}{\partial s}\Phi(t,\alpha(s))\rangle\\
&=&2\psi_{\xi}\varphi(t,s),
\end{eqnarray*}
so that $\varphi(t,s)=e^{2\psi_{\xi}t}\varphi(0,s)=0$. Therefore, the closed conformal character of $\xi$ assures that all images $t\mapsto\Phi(t,\Xi)$, $|t|<\epsilon$, are leaves of $\Xi^{\bot}$, hence totally umbillic in $\tilde M$.

On the other hand, since $\frac{D}{\partial t}\xi(\Phi(t,\alpha(s)))=\psi_{\xi}\xi(\Phi(t,\alpha(s)))$, the curves $t\mapsto\Phi(t,\alpha(s))$ are pre-geodesics in $\Omega=\Phi((-\epsilon,\epsilon)\times\Xi)$. Since $n\geq 2$, it follows from (a) that $\Xi$ is simply connected, and the local version of a theorem of Hiepko (see Theorem $1$ and Corollary $1$ of~\cite{Ponge:91}) assures that the metric of $\tilde M$ on $\Omega$ is warped, say $\Omega$ isometric to $(-\epsilon,\epsilon)\times_f\Xi$.

Therefore, Lemma~\ref{lemma:ambient curvature along planes containing xi}, together with Proposition $7.42$ $(2)$ of~\cite{O'Neill:83}, gives, for $X,Y\in\mathfrak X(\Xi)$,
$$0=R_{\tilde M}(\xi,X)Y=(f''+\psi_{\xi}f')\frac{\langle X,Y\rangle}{f}\xi,$$
and it follows from $f(0)=1$ that $f(t)=e^{-\psi_{\xi}t}$.
\end{proof}

A slight variant of the proof of item (a) of the above Theorem shows that we can equally treat the case of more homothetic vector fields, thus obtaining the following version.

\begin{theorem}
 Let $n,k\in\mathbb N$ be given, with $n\geq 2$, and $\tilde M^{n+k}$ be a Riemannian manifold possessing $k+1$ homothetic vector fields $\xi_1,\ldots,\xi_k,\eta$, with at least one of the $\xi_i$ being non-parallel, and $\Xi$ be a complete leaf of $\langle\xi_1,\ldots,\xi_k\rangle^{\bot}$. If $\Xi$ is an Einstein manifold with positive Ricci curvature, and $\{\xi_1,\ldots,\xi_k,\eta\}$ is linearly independent in at least one point of $\Xi$, then $\Xi$ is isometric to an Euclidean $n-$sphere.
\end{theorem}

\section{On a result of J. Simons}\label{section:On a result of J. Simons}

In this section we generalize a thorem of J. Simons~\cite{Simons:68}, which shows how one can build isometric immersions with parallel mean curvature in $\mathbb R^{n+k+1}$ from minimal immersions $\varphi:M^n\rightarrow\mathbb S^{n+k}$.

As before, let $\tilde M^{n+k+1}$ be an $(n+k+1)-$dimensional Riemannian manifold, furnished with a closed conformal vector field $\xi$ having conformal factor $\psi_{\xi}$. If $\xi\neq 0$ on $\overline M$, we saw in the previous section that the distribution $\xi^{\bot}$ of vector fields orthogonal to $\xi$ is integrable, with leaves totally umbilical in $\tilde M^{n+k+1}$.

Let $\Xi^{n+k}$ be such a leaf and $\varphi:M^n\rightarrow\Xi^{n+k}$ be an isometric immersion, where $M^n$ is a compact Riemannian manifold. If $\Psi$ denotes the flow of $\xi$, the compactness of $M^n$ guarantees the existence of $\epsilon>0$ such that $\Psi$ is defined on $(-\epsilon,\epsilon)\times\varphi(M)$, and the map
\begin{equation}\label{eq:immersion Phi}
 \begin{array}{rccc}
 \Phi:&(-\epsilon,\epsilon)\times M^n&\longrightarrow&\tilde M^{n+k+1}\\
&(t,q)&\mapsto&\Psi(t,\varphi(q))
  \end{array}
\end{equation}
is also an immersion. Furnishing $(-\epsilon,\epsilon)\times M^n$ with the metric induced by $\Phi$, we turn $\Phi$ into an isometric immersion such that $\Phi_{|\{0\}\times M^n}=\varphi$. 

Finally, letting $\text{Ric}_{\tilde M}$ denote the field of self-adjoint operadors associated to the Ricci tensor of $\tilde M$, we get the above-mentioned generalization of Simons' result (see also~\cite{Caminha:091} for the case of conformally stationary Lorentz manifolds).

\begin{theorem}
In the above notations, let $\psi_{\xi}\neq 0$ on $\varphi(M)$. If $\tilde M$ has constant sectional curvature or $\text{\rm Ric}_{\tilde M}(\xi)=0$, then the following are equivalent:
\begin{enumerate}
 \item[$(a)$] $\varphi$ is minimal.
 \item[$(b)$] $\Phi$ is minimal.
 \item[$(c)$] $\Phi$ has parallel mean curvature.
\end{enumerate}
\end{theorem}
 
\begin{proof}[\bf Prova]
Fix $p\in M$ and, on a neighborhood $\Omega$ of $p$ in $M$, an orthonormal frame $\{e_1,\ldots,e_n,\eta_1,\ldots,\eta_k\}$ adapted to $\varphi$, such that $\{e_1,\ldots,e_n\}$ is geodesic at $p$. 

If $E_1,\ldots,E_n,N_1,\ldots,N_k$ are the vector fields on $\Psi((-\epsilon,\epsilon)\times\Omega)$ obtained from the $e_i$'s and $\eta_{\beta}$'s by parallel transport along the integral curves of $\xi$ that intersect $\Omega$, it follows that $\{E_1,\ldots,E_n,\frac{\xi}{|\xi|},N_1,\ldots,N_k\}$ is an orthonormal frame on $\Psi((-\epsilon,\epsilon)\times\Omega)$, adapted to the immersion (\ref{eq:immersion Phi}).

Let $\tilde\nabla$ be the Levi-Civita connections of $\tilde M$ and $\tilde H$ the mean curvature vector of $\Phi$. It follows from the closed conformal character of $\xi$ that, on $\Phi((-\epsilon,\epsilon)\times\Omega)$,
\begin{equation}\label{eq:auxiliar 0 para cone com vetor curvatura media paralelo}
\tilde H=\frac{1}{n+1}(\tilde\nabla_{E_i}E_i+\tilde\nabla_{\xi/|\xi|}(\xi/|\xi|))^{\bot}
=\frac{1}{n+1}(\tilde\nabla_{E_i}E_i)^{\bot}, 
\end{equation}
where $\bot$ denotes orthogonal projection on $T\Phi((-\epsilon,\epsilon)\times\Omega)^{\bot}$.

In order to compute $\tilde\nabla_{E_i}E_i$ along the integral curve that passes through $p$, note that
\begin{equation}\label{eq:auxiliar 2 para cone com vetor curvatura media paralelo}
 \langle\tilde\nabla_{E_i}E_i,\xi\rangle=-\langle E_i,\tilde\nabla_{E_i}\xi\rangle=-n\psi_{\xi}.
\end{equation}

Now, if $\tilde R$ stands for the curvature operator of $\tilde M$, observe that
\begin{equation}\label{eq:auxiliar 3 para cone com vetor curvatura media paralelo}
 \begin{split}
 \frac{d}{dt}\langle\tilde\nabla_{E_i}E_i,E_k\rangle&\,=\langle\tilde\nabla_{\xi}\tilde\nabla_{E_i}E_i,E_k\rangle\\
&\,=\langle\tilde R(\xi,E_i)E_i,E_k\rangle+\langle\tilde\nabla_{E_i}\tilde\nabla_{\xi}E_i,E_k\rangle+\langle\tilde\nabla_{[\xi,E_i]}E_i,E_k\rangle\\
&\,=\langle\text{Ric}_{\tilde M}(\xi),E_k\rangle-\langle\tilde\nabla_{\tilde\nabla_{E_i}\xi}E_i,E_k\rangle\\
&\,=-\psi_{\xi}\langle\tilde\nabla_{E_i}E_i,E_k\rangle. 
 \end{split}
\end{equation}
Note that, in the last equality, we used the fact that either $\tilde M$ has constant sectional curvature or $\text{Ric}_{\tilde M}(\xi)=0$ to conclude that $\langle\text{Ric}_{\tilde M}(\xi),E_k\rangle=0$.

Let $D$ and $\nabla$ respectively denote the Levi-Civita connections of $\Xi^{n+k}$ and $M^n$. Since $\{e_1,\ldots,e_n\}$ is geodesic at $p$ (on $M$), it follows that
\begin{equation}\label{eq:auxiliar 4 para cone com vetor curvatura media paralelo}
 \langle\tilde\nabla_{E_i}E_i,E_k\rangle_p=\langle D_{e_i}e_i,e_k\rangle_p=\langle(D_{e_i}e_i)^{\bot}+\nabla_{e_i}e_i,e_k\rangle_p=0.
\end{equation}
Therefore, solving the Cauchy problem formed by (\ref{eq:auxiliar 3 para cone com vetor curvatura media paralelo}) and (\ref{eq:auxiliar 4 para cone com vetor curvatura media paralelo}), we get
\begin{equation}\label{eq:auxiliar 5 para cone com vetor curvatura media paralelo}
\langle\tilde\nabla_{E_i}E_i,E_k\rangle_{\Psi(t,p)}=0,\,\,\forall\,\,|t|<\epsilon. 
\end{equation}

Analogously to (\ref{eq:auxiliar 3 para cone com vetor curvatura media paralelo}), we get
\begin{equation}\label{eq:auxiliar 6 para cone com vetor curvatura media paralelo}
  \frac{d}{dt}\langle\tilde\nabla_{E_i}E_i,N_{\beta}\rangle=-\psi_{\xi}\langle\tilde\nabla_{E_i}E_i,N_{\beta}\rangle.
\end{equation}
On the other hand, letting $A_{\beta}:T_pM\rightarrow T_pM$ denote the shape operator of $\varphi$ in the direction of $\eta_{\beta}$ and writing $A_{\beta}e_i=h^{\beta}_{ij}e_j$, we have
\begin{equation}\label{eq:auxiliar 7 para cone com vetor curvatura media paralelo}
 \langle\tilde\nabla_{E_i}E_i,N_{\beta}\rangle_p=\langle D_{e_i}e_i,\eta_{\beta}\rangle_p=\langle A_{\beta}e_i,e_i\rangle_p=h^{\beta}_{ii}.
\end{equation}
Solving the Cauchy problem formed by (\ref{eq:auxiliar 6 para cone com vetor curvatura media paralelo}) and (\ref{eq:auxiliar 7 para cone com vetor curvatura media paralelo}), we get
\begin{equation}\label{eq:auxiliar 8 para cone com vetor curvatura media paralelo}
 \langle\tilde\nabla_{E_i}E_i,N_{\beta}\rangle_{\Psi(t,p)}=h^{\beta}_{ii}\exp\left(-\int_0^t\psi_{\xi}(s)ds\right). 
\end{equation}

It finally follows from (\ref{eq:auxiliar 2 para cone com vetor curvatura media paralelo}), (\ref{eq:auxiliar 5 para cone com vetor curvatura media paralelo}) and (\ref{eq:auxiliar 8 para cone com vetor curvatura media paralelo}) that, at the point $(t,p)$,
\begin{eqnarray*}
 \tilde\nabla_{E_i}E_i&=&\langle\tilde\nabla_{E_i}E_i,E_k\rangle E_k+\langle\tilde\nabla_{E_i}E_i,\xi\rangle\frac{\xi}{|\xi|^2}+\langle\tilde\nabla_{E_i}E_i,N_{\beta}\rangle N_{\beta}\\
&=&-n\frac{\psi_{\xi}}{|\xi|^2}\xi+\exp\left(-\int_0^t\psi_{\xi}(s)ds\right)h^{\beta}_{ii}N_{\beta}.
\end{eqnarray*}
Therefore, (\ref{eq:auxiliar 0 para cone com vetor curvatura media paralelo}) gives us
\begin{equation}\label{eq:auxiliar 9 para cone com vetor curvatura media paralelo}
 \tilde H=\frac{1}{n+1}\exp\left(-\int_0^t\psi_{\xi}(s)ds\right)h^{\beta}_{ii}N_{\beta}.
\end{equation}

Let us finally establish the equivalence of (a), (b) and (c), observing that (b) $\Rightarrow$ (c) is always true.\\

\noindent (a) $\Rightarrow$ (b): if $\varphi$ is minimal, we have $h^{\beta}_{ii}=0$ for all $1\leq\beta\leq k$, and it follows from (\ref{eq:auxiliar 9 para cone com vetor curvatura media paralelo}) that $\tilde H=0$.\\

\noindent (c) $\Rightarrow$ (a): if $\tilde\nabla^{\bot}\tilde H=0$, then, along the integral curve of $\xi$ that passes through $p$, the parallelism of the $N_{\beta}$ gives
$$0=\tilde\nabla_{\xi}^{\bot}\tilde H=\left(\frac{D\tilde H}{dt}\right)^{\bot}=-\frac{\psi_{\xi}(t)}{n+1}\exp\left(-\int_0^t\psi_{\xi}(s)ds\right)h^{\beta}_{ii}N_{\beta}.$$
However, since $\psi_{\xi}\neq 0$ on $\varphi(M)$, it follows from the above equality that $h^{\beta}_{ii}=0$ at $p$ for all $1\leq\beta\leq k$, so that $\varphi$ is minimal at $p$.
\end{proof}

\begin{example}
Let $\tilde M^{n+k+1}=I\times_tF^{n+k}$, so that $\xi=t\partial_t$ is closed and conformal, with conformal factor $\psi_{\xi}=1$. The leaves of $\xi^{\bot}$ are the slices $\Xi^{n+k}=\{t_0\}\times F^{n+k}$, and it follows from Corollary $7.43$ of {\rm\cite{O'Neill:83}} that $\text{Ric}_{\tilde M}(\xi)=0$. Therefore, by the previous theorem, an isometric immersion $\varphi:M^n\rightarrow\Xi^{n+k}$ is minimal if and only if the canonical immersion of $I\times_tM^n$ into $\tilde M^{n+k+1}$ is also minimal.

In particular, taking $I=(0,+\infty)$ and $F^{n+k}=\mathbb S^{n+k}$ we get the classical Simons' result which we alluded to in the beginning of this section.
\end{example}

\begin{example}
Let $I\subset(0,+\infty)$ and $F^{n+k}$ be a flat Riemannian manifold, so that the warped product $\tilde M^{n+k+1}=I\times_{e^t}F^{n+k}$ has constant sectional curvature identically equal to $-1$. If $\varphi:M^n\rightarrow\{t_0\}\times F^{n+k}$ is an isometric immersion, then $\varphi$ is minimal if and only if 
the canonical immersion $\Phi$ of $I\times_{e^t}M^n$ into $I\times_{e^t}F^{n+k}$ is also minimal. 

In particular, if $I=\mathbb R$ and $F^{n+k}=\mathbb R^{n+k}$, we already know that $\tilde M^{n+k+1}=\mathbb H^{n+k+1}$, the $(n+k+1)-$dimensional hyperbolic space. In the half-space model, if $\Xi^{n+k}$ is the horosphere $\{x_{n+k+1}=a\}$, $a>0$, then an isometric immersion $\varphi:M^n\rightarrow\Xi^{n+k}$ is minimal if and only if the union of the vertical geodesics of $\mathbb H^{n+k+1}$ passing through points of $\varphi(M)$ is minimal in $\mathbb H^{n+k+1}$.
\end{example}

\begin{example}
Let $I\subset(0,\pi)$ and $F^{n+k}$ be a Riemannian manifold of sectional curvature $1$, so that the warped product $\tilde M^{n+k+1}=I\times_{\sin t}F^{n+k}$ has constant sectional curvature identically equal to $1$. If $\varphi:M^n\rightarrow\{t_0\}\times F^{n+k}$ is an isometric immersion, then $\varphi$ is minimal if and only if 
the canonical immersion $\Phi$ of $I\times_{\sin t}M^n$ into $I\times_{\sin t}F^{n+k}$ is also minimal. 

In particular, if $I=(0,\pi)$ and $F^{n+k}=\mathbb S^{n+k}$, we already know that $\tilde M^{n+k+1}=\mathbb S^{n+k+1}$, the $(n+k+1)-$dimensional Euclidean sphere. If $\Xi^{n+k}$ is a geodesic sphere of $\mathbb S^{n+k+1}$ centered at $p$, then an isometric immersion $\varphi:M^n\rightarrow\Xi^{n+k}$ is minimal if and only if the union of the great circles of $\mathbb S^{n+k+1}$ passing through $-p$, $p$ and points of $\varphi(M)$ is minimal in $\mathbb S^{n+k+1}$.
\end{example}

\section{Bernstein-type theorems}\label{section:Bernstein-type theorems}

We continue to employ the notations above. From now on, we let $x:M^n\rightarrow\tilde M^{n+1}$ be a connected, complete, oriented hypersurface transversal to $\xi$ at every point, $N$ be a unit normal vector field which orients $M$, and $A$ and $H$ respectively the second fundamental form and the mean curvature of $M$ with respect to $N$.

If $f_{\xi}:M\rightarrow\mathbb R$ is given by $f_{\xi}=\langle\xi,N\rangle$, the transversality condition above, together with the connectedness of $M$, give that $f_{\xi}$ is either positive or negative on $M$. On the other hand, standard computations (cf.~\cite{Barros:08}) give
\begin{equation}\label{eq:Gradient of f xi}
 \nabla f_{\xi}=-A(\xi^{\top})
\end{equation}
and
\begin{equation}\label{eq:Laplacian of f xi}
\Delta f_{\xi}=-n\xi^{\top}(H)-(\text{\rm Ric}_{\tilde M}(N,N)+|A|^2)f_{\xi}-n\left\{H\psi_{\xi}+N(\psi_{\xi})\right\},
\end{equation}
where $(\,\,)^{\top}$ stands for orthogonal projections onto $M$.

If $\eta$ is another closed conformal vector field on $\tilde M$ and $g:M\rightarrow\mathbb R$ is given by $g=\langle\xi,\eta\rangle$, then more standard computations give
\begin{equation}\label{eq:Gradient of g}
\nabla g=\psi_{\xi}\eta^{\top}+\psi_{\eta}\xi^{\top}.
\end{equation}
The computation of $\Delta g$ is a bit more involving and, as far as we know, not found in the literature, so we present it for the sake of completeness.

\begin{lemma}\label{lemma:Laplacian of xi inner eta}
Under the above notations, we have
\begin{equation}\label{eq:Laplacian of xi inner eta}
 \Delta g=\eta^{\top}(\psi_{\xi})+\xi^{\top}(\psi_{\eta})+nH(\psi_{\xi}f_{\eta}+\psi_{\eta}f_{\xi})+2n\psi_{\xi}\psi_{\eta}.
\end{equation}
\end{lemma}

\begin{proof}
Let $\nabla$ be the Levi-Civita connection of $M$. For fixed $p\in M$, choose an orthonormal frame field $\{e_j\}$ on a neighborhood of $p$ in $M$, such that $(\nabla_{e_i}e_j)(p)=0$ for all $i,j$. Then, at $p$ we have
\begin{eqnarray*}
 \Delta g&=&\sum_je_j(e_j(g))=\sum_je_j\langle\psi_{\xi}e_j,\eta\rangle +\sum_je_j\langle\xi,\psi_{\eta}e_j\rangle\\
&=&\sum_j\{\langle e_j(\psi_{\xi})e_j,\eta\rangle+\psi_{\xi}e_j\langle e_j,\eta\rangle\}+\sum_j\{\langle\xi,e_j(\psi_{\eta})e_j\rangle+\psi_{\eta}e_j\langle\xi,e_j\rangle\}\\
&=&\langle\tilde\nabla\psi_{\xi}-N(\psi_{\xi})N,\eta\rangle+\sum_j\{\psi_{\xi}\langle\tilde\nabla_{e_j}e_j,\eta\rangle+\psi_{\xi}\langle e_j,\tilde\nabla_{e_j}\eta\rangle\}\\
&&+\langle\xi,\tilde\nabla\psi_{\eta}-N(\psi_{\eta})N\rangle+\sum_j\{\psi_{\eta}\langle\tilde\nabla_{e_j}\xi,e_j\rangle+\psi_{\eta}\langle\xi,\tilde\nabla_{e_j}e_j\rangle\}\\
&=&\eta(\psi_{\xi})-N(\psi_{\xi})f_{\eta}+\sum_j\{\psi_{\xi}\langle Ae_j,e_j\rangle f_{\eta}+\psi_{\xi}\langle e_j,\psi_{\eta}e_j\rangle\}\\
&&+\xi(\psi_{\eta})-N(\psi_{\eta})f_{\xi}+\sum_j\{\psi_{\eta}\langle\psi_{\xi}e_j,e_j\rangle+\psi_{\eta}f_{\xi}\langle e_j,Ae_j\rangle\}\\
&=&(\eta-f_{\eta}N)(\psi_{\xi})+\sum_j\{\psi_{\xi}\langle Ae_j,e_j\rangle f_{\eta}+\psi_{\xi}\langle e_j,\psi_{\eta}e_j\rangle\}\\
&&+(\xi-f_{\xi}N)(\psi_{\eta})+\sum_j\{\psi_{\eta}\langle\psi_{\xi}e_j,e_j\rangle+\psi_{\eta}f_{\xi}\langle e_j,Ae_j\rangle\}\\
&=&\eta^{\top}(\psi_{\xi})+\xi^{\top}(\psi_{\eta})+nH(\psi_{\xi}f_{\eta}+\psi_{\eta}f_{\xi})+2n\psi_{\xi}\psi_{\eta}.
\end{eqnarray*}

\end{proof}

As a consequence of the above computations and Proposition~\ref{prop:first corollary of Yau 76}, one has the following Bernstein-type general theorems for hypersurfaces, the first of which not requiring constant mean curvature.

\begin{theorem}\label{thm:Bernstein-type}
Let $\tilde M$ have nonnegative Ricci curvature, $\xi$ and $\eta$ be respectively a parallel and a homothetic
nonparallel vector field on $\tilde M^{n+1}$, and $x:M^n\rightarrow\tilde M^{n+1}$ be as above. If $|A|$
is bounded, $|\xi^{\top}|$ is integrable and $H$ doesn't change sign on $M$, then:
\begin{enumerate}
\item[$(a)$] $M$ is totally geodesic and the Ricci curvature of $\tilde M$ in the direction of $N$ vanishes identically.
\item[$(b)$] If $M$ is noncompact and ${\rm Ric}_M$ is also nonnegative, then $x(M)$ is contained in a leaf of $\xi^{\bot}$.
\end{enumerate}
\end{theorem}

\begin{proof}
Since $\xi$ is parallel and $\eta$ is homothetic and nonparallel, it follows from (\ref{eq:Gradient of f xi}), (\ref{eq:Gradient of g}), (\ref{eq:Laplacian of f xi}) and (\ref{eq:Laplacian of xi inner eta}) that $\nabla f_{\xi}=-A(\xi^{\top})$, $\nabla g=\psi_{\eta}\xi^{\top}$,
\begin{equation}\label{eq:Laplacian of f xi for parallel xi}
 \Delta f_{\xi}=-n\xi^{\top}(H)-(\text{\rm Ric}_{\tilde M}(N,N)+|A|^2)f_{\xi},
\end{equation}
and
$$\Delta g=nH\psi_{\eta}f_{\xi},$$
with $\psi_{\eta}$ constant and nonzero. Therefore, the hypothesis $|\xi^{\top}|\in\mathcal L^1(M)$ guarantees that
$|\nabla g|\in\mathcal L^1(M)$, and the hypothesis on $H$, together with the fact that $|f_{\xi}|>0$ on $M$ (which is in turn due to the transversality of $M$ and $\xi$), assures that $\Delta g$ is either nonnegative or nonpositive on $M$. Therefore, Corollary $1$ of~\cite{Yau:76} gives $\Delta g=0$ on $M$, and hence $H=0$ on $M$.

We now look at (\ref{eq:Laplacian of f xi for parallel xi}), which resumes to
$$\Delta f_{\xi}=-(\text{\rm Ric}_{\tilde M}(N,N)+|A|^2)f_{\xi},$$
and hence doesn't change sign on $M$ too. We also note that the boundedness of $|A|$ on $M$ gives
$$|\nabla f_{\xi}|\leq|A||\xi^{\top}|\in\mathcal L^1(M).$$
As before, these facts give $\Delta f_{\xi}=0$ on $M$, so that
$$\text{\rm Ric}_{\tilde M}(N,N)+|A|^2=0$$
on $M$. Since $\text{\rm Ric}_{\tilde M}(N,N)\geq 0$, we then get $\text{\rm Ric}_{\tilde M}(N,N)=0$ and $A=0$ on $M$, i.e.,
$M$ is totally geodesic. This finishes the proof of $(a)$.

Concerning $(b)$, $A=0$ on $M$ gives $\nabla f_{\xi}=0$ on $M$, so that $f_{\xi}=\langle\xi,N\rangle$ is constant on $M$ and
nonzero, due to the transversality assumption. However, $|\xi|^2$ is constant on $\tilde M$ (since $\xi$ is parallel) and
$$|\xi^{\top}|^2=|\xi|^2-\langle\xi,N\rangle^2,$$
so that $|\xi^{\top}|$ is also constant on $M$. Therefore,
$$+\infty>\int_M|\xi^{\top}|dM=|\xi^{\top}|\,{\rm Vol}(M).$$
But since $M$ is noncompact and has nonnegative Ricci curvature, another theorem of Yau
(Theorem $7$ of~\cite{Yau:76}) gives ${\rm Vol}(M)=+\infty$, and hence the only possibility is $|\xi^{\top}|=0$.
Therefore, Cauchy-Schwarz inequality gives that $\xi$ is parallel to $N$, and $x(M)$ is contained in a leaf of $\xi^{\bot}$.
\end{proof}

We remark that the above result generalizes one of the main results of~\cite{Caminha:09}. 

\begin{theorem}\label{thm:Bernstein-type 2}
Let $\tilde M$ have nonnegative Ricci curvature, $\xi$ be a homothetic vector field on $\tilde M^{n+1}$, and $x:M^n\rightarrow\tilde M^{n+1}$ be as before. If $|A|$ is bounded, $|\xi^{\top}|$ is integrable and $H$ is constant on $M$, then $M$ is totally umbilical and the Ricci curvature of $\tilde M$ in the direction of $N$ vanishes identically.
\end{theorem}

\begin{proof}
 Since $H$ is constant on $M$ and $\psi_{\xi}$ is constant on $\tilde M$, (\ref{eq:Laplacian of f xi}) reduces to
$$\Delta f_{\xi}=-(\text{\rm Ric}_{\tilde M}(N,N)+|A|^2)f_{\xi}-nH\psi_{\xi}.$$
A straightforward computation now gives
\begin{equation}\label{eq:pulo do gato 1}
{\rm div}_M(\xi^{\top})=n\psi_{\xi}+nHf_{\xi},
\end{equation}
so that
\begin{eqnarray}\label{eq:divergence of correct field}
{\rm div}_M(\nabla f_{\xi}+H\xi^{\top})&=&\Delta f_{\xi}+nH\psi_{\xi}+nH^2f_{\xi}\\
&=&-(\text{\rm Ric}_{\tilde M}(N,N)+|A|^2-nH^2)f_{\xi}.\nonumber
\end{eqnarray}

Since $\xi$ is transversal to $M$, ${\rm Ric}_{\tilde M}(N,N)\geq 0$ and
$|A|^2\geq nH^2$ by Cauchy-Schwarz inequality (with equality if and only if $M$ is totally umbilical), this last expression does not change sign on $M$. Now observe that
$$|\nabla f_{\xi}+H\xi^{\top}|=|-A\xi^{\top}+H\xi^{\top}|\leq(|A|+H)|\xi^{\top}|\in\mathcal L^1(M),$$
so that Proposition~\ref{prop:criterion for parallelism} gives ${\rm div}_M(\nabla f_{\xi}+H\xi^{\top})=0$ on $M$. Back to (\ref{eq:divergence of correct field}), we then get $\text{\rm Ric}_{\tilde M}(N,N)=0$ and $|A|^2-nH^2=0$, and we already mentioned that this last condition implies $M$ to be totally umbilical.
\end{proof}

If $\tilde M=I\times_tF^n$, then Corollary $7.36$ and Corollary $7.43$ $(1)$ of~\cite{O'Neill:83} respectively guarantee that $\Xi_{t_0}=\{t_0\}\times F^n$ is totally umbilic in $\tilde M$ and the Ricci curvature of $\tilde M$ in the direction of $\partial_t$ indeed vanishes. Therefore, the previous result yields the following corollary on warped products.

\begin{corollary}\label{coro:Bernstein-type for warped products}
Let $I\subset\mathbb R$ be an open interval, $F$ be an $n-$dimensional, complete oriented Riemannian manifold having nonnegative Ricci curvature, $\tilde M=I\times_tF^n$ and $x:M^n\rightarrow\tilde M^{n+1}$ be as before. If $|A|$ is bounded, $|\xi^{\top}|$ is integrable and $H$ is constant on $M$, then $M$ is totally umbilical and the Ricci curvature of $\tilde M$ in the direction of $N$ vanishes identically. In particular, if $F$ is closed and has positive Ricci curvature everywhere, then $x(M)\subset\{t_0\}\times F$, for some $t_0\in I$.
\end{corollary}

\begin{proof}
 The first part follows from the theorem. To the second one, if $F$ is closed and has positive Ricci curvature everywhere, then, according to the previous result, the only possible direction for $N$ is that of $\xi=t\partial_t$. But if $N$ is parallel to $\partial_t$, it is easy to see that $x(M)$ cannot jump from one leaf $\{t_0\}\times F$ to another (remember that $M$ is connected!).
\end{proof}

\begin{remark} 
Concerning the above corollary, we point out that 
$$|\xi^{\top}|=|t|\sqrt{1-\langle N,\partial_t\rangle^2}.$$
Therefore, the condition $|\xi^{\top}|\in\mathcal L^1(M)$ intuitively amounts to $N$ becoming almost parallel to $\partial_t$ as we go to infinity on $M$.
\end{remark}

As a special case of the previous corollary, we get a generalization of a theorem of J. Jellett~\cite{Jellett:1853} to complete cmc radial graphs over a finitely punctured sphere in Euclidean space.

\begin{corollary}\label{coro:Jellett's theorem over the punctured sphere}
Let $x:M^n\rightarrow\mathbb R^{n+1}$ be an embedding, such that $x(M)$ is a complete radial graph over the standard $n-$sphere $\mathbb S^n$ minus $k\geq 0$ points. If $|A|$ is bounded, $H$ is constant and $p\mapsto|x(p)^{\top}|$ is integrable on $M$, then $k=0$ and $x(M)$ is a sphere.
\end{corollary}

\begin{remark}
 If $\xi$ is a homothetic vector field on $\tilde M$, then $\xi$ gives rise to a non-parallel calibration on $\tilde M$ in the sense of R. Harvey and H. B. Lawson {\rm(}cf.~\cite{Harvey:82}{\rm)}. In fact, the calibration is $\frac{1}{|\xi|}\omega$, where $\omega$ is the dual $1-$form of $\xi$. In this respect, we remark that G. Li and I. Salavessa~\cite{Li:09} obtained some Bernstein-type results related to ours above {\rm(}although, to the best of our knowledge, not the same ones{\rm)} by studying Bernstein-type results in manifolds possessing calibrations.
\end{remark}

We now turn to the case of $\xi$ conformal but not necessarily closed, so that the orthogonal distribution $\xi^{\bot}$ is not necessarily integrable. Since $\psi_{\xi}$ is not necessarily constant, we are obliged to search for an adequate $\xi$.

To this end, let $\tilde M$ have constant sectional curvature, say $c$. Following A. Barros and P. Sousa in~\cite{Barros:082}, fix $p\in\tilde M$ and let $r:\tilde M\rightarrow[0,+\infty)$ be the distance function from $p$ on $\tilde M$. If $s:\mathbb R\rightarrow\mathbb R$ is the solution to the ODE $y''+cy=0$, $y(0)=0$, $y'(0)=1$, then
$$\xi=(s\circ r)\tilde\nabla r$$
is conformal, with conformal factor $\psi_{\xi}=s'\circ r$. Therefore,
\begin{eqnarray*}
N(\psi_{\xi})&=&\langle N,\tilde\nabla\psi_{\xi}\rangle=\langle N,\tilde\nabla(s'(r))\rangle=
\langle N,s''(r)\tilde\nabla r\rangle\\
&=&-\langle N,cs(r)\tilde\nabla r\rangle=-c\langle N,\xi\rangle=-cf_{\xi}.
\end{eqnarray*}

We now observe that formula (\ref{eq:Laplacian of f xi}) remains valid for $\xi$ conformal but not necessarily closed (cf.~\cite{Barros:08}), and hence (recall that ${\rm Ric}_{\tilde M}(N,N)=cn$)
$$\Delta f_{\xi}+nH\psi_{\xi}=-|A|^2f_{\xi}.$$
On the other hand, it is easy to check that (\ref{eq:pulo do gato 1}) also remains valid for $\xi$ conformal but not necessarily closed, and it follows as in the proof of theorem~\ref{thm:Bernstein-type 2} that
\begin{equation}\label{eq:divergence of correct field 2}
{\rm div}_M(\nabla f_{\xi}+H\xi^{\top})=-(|A|^2-nH^2)f_{\xi}.
\end{equation}

Taking pieces together we get the following generalization of the main theorems in~\cite{Barros:082}:

\begin{theorem}\label{thm:Bernstein-type 3}
Let $\tilde M^{n+1}$ have constant sectional curvature and $x:M^n\rightarrow\tilde M^{n+1}$ be as before and such that $x(M)\cap {\rm Cut}(p)=\emptyset$ for some $p\in\tilde M$. If $|A|$ is bounded, $|\xi^{\top}|$ is integrable and $H$ is constant on $M$, then $M$ is totally umbilical.
\end{theorem}

\begin{proof}
Fix $q\in M$ and let $\{e_j\}$ be an orthonormal frame field on a neighborhood of $q$ in $M$. For our particular choice of $\xi$, equation $(3.1)$ of~\cite{Barros:082} gives
\begin{eqnarray*}
e_j\langle\xi,N\rangle&=&\langle\tilde\nabla_{e_j}\xi,N\rangle+\langle\tilde\xi,\nabla_{e_j}N\rangle\\
&=&\langle\tilde\nabla_{e_j}((s\circ r)\tilde\nabla r),N\rangle-\langle\xi,Ae_j\rangle\\
&=&(s\circ r)\langle\tilde\nabla_{e_j}\tilde\nabla r,N\rangle+(s'\circ r)e_j(r)\langle\tilde\nabla r,N\rangle-\langle A(\xi^{\top}),e_j\rangle\\
&=&(s\circ r)\cdot\frac{s'\circ r}{s\circ r}\{\langle e_j,N\rangle-\langle e_j,\tilde\nabla r\rangle\langle N,\tilde\nabla r\rangle\}\\
&&+(s'\circ r)\langle\tilde\nabla r,e_j\rangle\langle\tilde\nabla r,N\rangle-\langle A(\xi^{\top}),e_j\rangle\\
&=&(s\circ r)\cdot\frac{s'\circ r}{s\circ r}\{\langle e_j,N\rangle-\langle e_j,\tilde\nabla r\rangle\langle N,\tilde\nabla r\rangle\}\\
&&+(s'\circ r)\langle\tilde\nabla r,e_j\rangle\langle\tilde\nabla r,N\rangle-\langle A(\xi^{\top}),e_j\rangle\\
&=&-\langle A(\xi^{\top}),e_j\rangle.
\end{eqnarray*}
This way, we still have
$$\nabla f_{\xi}=-\sum_j\langle A(\xi^{\top}),e_j\rangle e_j=-A(\xi^{\top}).$$

We now just have to follow the steps in the proof of Theorem~\ref{thm:Bernstein-type 2}, using (\ref{eq:divergence of correct field 2}) this time and recalling that $M$ is transversal to $\xi$ by assumption.
\end{proof}

 The above result extends the main theorem of~\cite{Barros:082} in two directions, for it covers the compact {\rm(}closed{\rm)} case in $\mathbb H^{n+1}$ and the complete noncompact case in both $\mathbb S^{n+1}$ and $\mathbb H^{n+1}$. We explicitly state this extension for the hyperbolic space as the following

\begin{corollary}\label{coro:Jellett's theorem for spheres and hyperbolic}
Let $x:M^n\rightarrow\mathbb H^{n+1}$ be an embedding, such that $x(M)$ is a complete radial graph over a geodesic sphere minus $k\geq 0$ points. If $|A|$ is bounded, $H$ is constant and $p\mapsto|x(p)^{\top}|$ is integrable on $M$, then $k=0$ and $x(M)$ is a geodesic sphere.
\end{corollary}

\end{document}